\documentclass{amsart}
\usepackage{amssymb,latexsym,amscd,hyperref}
\usepackage{tikz-cd}
\usepackage{pb-diagram}
\theoremstyle{plain}
\newtheorem{theorem}{Theorem}
\newtheorem{corollary}[theorem]{Corollary}
\newtheorem{lemma}[theorem]{Lemma}
\newtheorem{example}[theorem]{Example}
\newtheorem{remark}[theorem]{Remark}
\newtheorem{proposition}[theorem]{Proposition}
\newtheorem{algorithm}[theorem]{Algorithm}
\newtheorem{conjecture}[theorem]{Conjecture}
\theoremstyle{definition}
\newtheorem{definition}[theorem]{Definition}

\theoremstyle{remark}

\numberwithin{theorem}{section} 
\newcommand {\Q}{{\mathbb{Q}}}
\newcommand {\C}{{\mathbb{C}}}
\newcommand {\R}{{\mathbb{R}}}
\newcommand {\Z}{{\mathbb{Z}}}
\newcommand {\N}{{\mathbb{N}}}
\newcommand {\F}{{\mathbb{F}}}
\newcommand {\PP}{\mathbb{P}}

\newcommand {\OO}{{\mathcal{O}}}

\newcommand {\ida}{{\mathfrak{a}}}
\newcommand {\idc}{{\mathfrak{c}}}
\newcommand {\idp}{{\mathfrak{p}}}

\newcommand {\idP}{{\mathfrak{P}}}
\newcommand {\fa}{\ida}
\newcommand{\cN}        {{\mathcal N}}
\newcommand{\Norm}        {{\mathcal N}}

\newcommand{\disc}     {{\mathfrak d}}%
\newcommand{\normal}     {\unlhd}%
\newcommand{\id}     {\mathop{\rm {id}}}%
\newcommand{\Gal}      {\mathop{\rm {Gal}}}

\newcommand{\Aut}      {\mathop{\rm {Aut}}}
\newcommand{\Zen}      {\mathop{\rm {Z}}}

\newcommand{\Cl}      {\mathop{\rm {Cl}}}

\newcommand{\ind}        {{\mathop{\rm ind}}}
\newcommand{\rk}        {{\mathop{\rm rk}}}

\begin{document}

\title{The asymptotics of nilpotent Galois groups}

\author{J\"urgen Kl\"uners}
\address{University Paderborn, Department of Mathematics,
Warburger Str. 100, 33098 Paderborn, Germany} 

\email{klueners@math.uni-paderborn.de}


\begin{abstract} 
  We prove an upper bound for the  asymptotics of counting functions of number fields with  nilpotent Galois groups.    
\end{abstract}

\maketitle
 
\renewcommand{\theenumi}{(\roman{enumi})}

\section{The counting function}
Let $G\le S_n$ be a finite transitive permutation group with $n>1$
and $k$ be a number field. We say that a finite extension
$K/k$ has Galois group $G$ if the normal closure $\hat K$ of $K/k$ has Galois
group isomorphic to $G$ and $K$ is the fixed field in $\hat K$ under a point
stabilizer of $G$. By abuse of notation we will write $\Gal(K/k)=G$ in this
situation. We let
$$Z(k,G;x):=|\left\{K/k\mid \Gal(K/k)=G,\ \cN_{k/\Q}(d_{K/k})\le
  x\right\}|$$ to be the number of field extensions of $k$ (inside a
fixed algebraic closure $\bar\Q$) of degree~$n$ with Galois group
permutation isomorphic to $G$ and norm of the discriminant $d_{K/k}$
bounded above by $x$.  Note that there are only finitely many number
fields of given degree and bounded discriminant, hence $Z(k,G;x)$
is finite for all $G, k, x$.  The goal of this note is to give the
asymptotics of $Z(k,G;x)$ for $x\rightarrow\infty$ for certain groups
$G$.
Gunter Malle \cite{Ma4,Ma5} has proposed a precise conjecture about the asymptotic 
behavior of the function $Z(k,G;x)$ for $x\rightarrow\infty$. In order
to state it, we introduce some group theoretic invariants of permutation
groups. 
\begin{definition}
  Let $G\leq S_n$ be a transitive subgroup acting on $\Omega=\{1,\ldots,n\}$.
  \begin{enumerate}
  \item For $g\in G$ we define the index $\ind(g):= n- \mbox{ the number of orbits of $g$ on }\Omega.$
  \item $\ind(G):=\min\{\ind(g): \id\ne g\in G\},\; a(G):=\ind(G)^{-1}$.
  \end{enumerate}
\end{definition} 
Since all elements in a conjugacy class $C$ of $G$ have the same index
we can define $\ind(C)$ in a canonical way. The absolute Galois group
of $k$ acts on the set of conjugacy classes of $G$ via the action on the
$\bar\Q$--characters of $G$. The orbits under this action
are called $k$-conjugacy classes.
\begin{definition}
  For a number field $k$ and a transitive subgroup $G\leq S_n$ we define:
  $$b(k,G):=\#\{C : C\; k\mbox{-conjugacy class of $G$ of minimal index }\ind(G)\}.$$
\end{definition}
Now we can state the conjecture of Malle \cite{Ma5}, where $f(x) \sim g(x)$ is a 
notation for $\lim_{x\rightarrow\infty}f(x)/g(x) =1$.
\begin{conjecture}\label{con}(Malle)
  For all number fields $k$ and all transitive permutation groups
  $G\leq S_n$ there exists a constant $c(k,G)>0$ such that
  $$Z(k,G;x) \sim c(k,G)x^{a(G)} \log(x)^{b(k,G)-1},$$
  where $a(G)$ and $b(k,G)$ are given as above.  
\end{conjecture}
This conjecture has been proved for abelian groups \cite{Wr}.
Assuming that $G$ is regular, i.e.~ we count normal nilpotent number fields, we have shown \cite[Thm. 6.3]{KlMa2}:
$$\lim_{x\rightarrow\infty} \frac{\log Z(k,G;x)}{\log x}= a(G).$$
This shows the correctness of the $a$-constant for those groups.

In the
following we use the notation $f(x)=O(g(x))$, if 
$\limsup_{x\rightarrow\infty} f(x)/g(x) <\infty$.
For all number fields $k$, all prime numbers $\ell$, and all $\ell$-groups $G$ it is shown in \cite[Cor. 7.3]{KlMa2} that
\[Z(k,G;x) = O(x^{a(G)+\epsilon})\text{ for all }\epsilon>0.  \]
%$$\limsup_{x\rightarrow\infty} \frac{\log Z(k,G;x)}{\log x}\leq a(G).$$
We remark that \cite[Cor. 1.8]{Alb} is able to prove this upper bound for all nilpotent groups in all transitive representations.

Unfortunately, there are counterexamples \cite{Kl5} with respect to $b(k,G)$. There are no known counterexamples for nilpotent groups.

The goal of this article is to extend the results of \cite{Alb,KlMa2}. In Theorem \ref{count2} we are able to replace the
upper bound of type $x^{a(G)+\epsilon}$ for all $\epsilon>0$ by $x^{a(G)}\log(x)^{d(k,G)-1}$. 
\begin{theorem} \label{countmain}
	Let $G\leq S_n$ be a transitive nilpotent group and $k$ be a number field. Then there exists a constant $d(k,G)\geq b(k,G)$ such
	that
	$$Z(k,G;x) = O(x^{a(G)} \log(x)^{d(k,G)-1}).$$
\end{theorem}
For a discussion of the constant $d(k,G)$ we refer to the discussion starting at Remark \ref{rem:d}. We show in Lemma \ref{nil} that $d(k,G)=b(k,G)$ for all nilpotent groups for which all elements of minimal index lie in the center. This trivially includes the abelian groups and the generalized quaternion groups. All elements of minimal index are of prime order $\ell$. We remark that $d(k,G)$ is only depending on the Sylow $\ell$-subgroup of $G$ and therefore we can construct many more examples as direct products of $p$-groups, for which our method yields the conjectured bound.  
 
%In general we have that
%$b(k,G)\leq d(k,G)$. In some cases we get equality, e.g. for all abelian and for generalized quaternion groups. For other cases we do not have equality. The reason is that we ignore embedding obstructions in our approach which leads to upper bounds being slightly weaker than expected.
%
%Nevertheless, as already mentioned we are able to replace the upper bounds of type $x^{a(G)+\epsilon}$ for all $\epsilon>0$ by an upper bound of type $x^{a(G)}\log(x)^{d(k,G)}$.
 These upper bounds are valid for all transitive nilpotent groups and we think that the proof of the results are much easier compared to the proofs given in \cite{KlMa2}.

We mention that there is some active research about these conjectures and we refer the reader to the
survey articles \cite{Bel,CoDiOl4, Wood}.

\section{Analytic tools}

In this section we develop some tools for the asymptotic behavior of
sums corresponding to Euler products.  In principle these tools are
known, but it is very difficult to find good references.  \cite[Lemma
5.4 and 5.5]{Sau} consider Euler products of the form
$Z(s) = \prod_p W(p,p^{-s})$, where $W\in\Z[x,y]$ is a polynomial with constant term
$1$. We only need the classical case that the $x$-degree of $W$ is $0$.
They prove the area of convergence and that there is a meromorphic continuation to
the left. Unfortunately, they do not give the order of the pole. Knowing these
properties is necessary to get the asymptotic behavior of the corresponding series.

We give a version of a Tauberian theorem which is useful for our applications.
This version is based on \cite{Del} and also proved in \cite[page 121]{Nar2}.

\begin{theorem}[Tauberian theorem of Delange]\label{tauber}
  Let
$$f(s) = \sum_{n=1}^{\infty} \frac{a_n}{n^s} $$ be a convergent Dirichlet series for $s\in\C$ with $\Re(s)>a>0$ with 
real $a_n\geq 0$. Furthermore, we assume in the area of convergence for a positive $z\in\R$:
$$f(s) = \frac{g(s)}{(s-a)^z}+h(s),$$
where $g$ and $h$ are analytic functions in the half plane
$\Re(s)\geq a$ with $g(a)\ne 0$.
Then we get for  $x\rightarrow \infty$:
$$ \sum_{n\leq x} a_n =\left(\frac{g(a)}{a\Gamma(z)}+o(1)\right)x^a\log(x)^{z-1},$$
where $\Gamma$ denotes the  Gamma--function. This is equivalent to
$$\sum_{n\leq x} a_n \sim \frac{g(a)}{a\Gamma(z)} x^a\log(x)^{z-1}.$$
\end{theorem}
We remark that \cite{Nar2} also proves a version where we get
$\log(\log(x))$-factors in the asymptotics. These cases do not occur
in this work. As an example the Riemann $\zeta$-function can be
written as $\zeta(s)=1/(s-1)+h(s)$, where $h(s)$ is analytic for
$\Re(s)>0$. In this special case the asymptotic behavior can be proved
without using the Tauberian theorem.
We need a similar result for number fields and we also want to restrict to
prime ideals with certain congruence properties.
Before we state the theorem we need the following lemma. In the following we use the notation $\PP$ for the prime numbers and the notation $\PP(k)$ for the (finite) prime ideals of the maximal order $\OO_k$ of $k$.
\begin{lemma}\label{euler2}
  Let $k$ be a number field, $\ell\in\PP$,
  $U:=\{\idp\in \PP(k)\mid \Norm(\idp)\equiv 0 \text{ \rm or }1 \bmod \ell\}$, and
  $\zeta_\ell$ be an $\ell$-th root of unity. For $d,m\in\N$ we define for $s\in\C$ with $\Re(s)>1/d$:
$$f(s):=\prod_{\idp\in U}\left(1+\frac{m}{\Norm(\idp)^{ds}}\right) = \sum_{n=1}^\infty \frac{a_n}{n^s}.$$
  Then we get with $e:=\frac{m}{[k(\zeta_\ell):k]}$ and a constant
  $c(k,\ell,d,m)>0$ for  $x\rightarrow \infty$:
$$\sum_{n\leq x} a_n \sim c(k,\ell,d,m)  x^{1/d} \log(x)^{e-1}.$$
%The Euler product is convergent for $\Re(s)>1/d$ and 
%The order of the pole at $s=1/d$ is $e$.
\end{lemma}
\begin{proof}
  We have the following identity for all $\idp\in U$:
  $$(1+m\Norm(\idp)^{-ds})(1-\Norm(\idp)^{-ds})^{m}
  =1-\binom{m+1}{2}\Norm(\idp)^{-2ds}+\cdots - m\Norm(\idp)^{-(m+1)d
    s}.$$ We get with an easy computation or by \cite[Lemma 5.4]{Sau} that
  $$g(s):=\prod_{\idp\in U}
  \left(1-\binom{m+1}{2}\Norm(\idp)^{-2ds}+\cdots -
    m\Norm(\idp)^{-(m+1)d s}\right)$$ is a convergent product for
  $\Re(s)\ge 1/d$.
  We have
  $$f(s)=\prod_{\idp\in U} \left(1+\frac{m}{\Norm(\idp)^{ds}}\right) = g(s) \prod_{\idp\in U} (1-\Norm(\idp)^{-ds})^{-m}.$$
  Now we write $U=U_0 \dot\cup U_1$, where $U_i$ contains the prime ideals which are congruent to $i \mod \ell$. Note that $U_0$ is a finite set and therefore 
  \[ g_0(s):=\prod_{\idp\in U_0} (1-\Norm(\idp)^{-ds})^{-m} \]
  is a convergent product for
  $\Re(s)\ge 1/d$.
  
  Let $n_\ell:=[k(\zeta_\ell):k]$, a divisor of $\ell-1$. All prime
  ideals in $U_1$ are totally split in $k(\zeta_\ell)$. Therefore we get
  for $\idp\in U_1$ the factorization
  $\idp\OO_{k(\zeta_\ell)}=\idP_1\cdots\idP_{n_\ell}$.  Since
  $\Norm(\idp)=\Norm(\idP_i)$ for $i=1,\ldots,n_{\ell}$ and using
  $e:=m/n_\ell$ we get for all $\idp\in U_1$:
  $$ (1-\Norm(\idp)^{-s})^{-m}= \prod_{\idP\mid \idp}(1-\Norm(\idP)^{-s})^{-e}.$$ 
  
  Let $\tilde U_1\subset\PP(k(\zeta_\ell))$ be the set of prime ideals lying
  over prime ideals of $U_1$. We get:
  $$\prod_{\idp\in U_1} (1-\Norm(\idp)^{-ds})^{-m} = \prod_{\idP\in \tilde U_1}(1-\Norm(\idP)^{-ds})^{-e} = \zeta_{k(\zeta_\ell)}(ds)^e\prod_{\idP\notin \tilde U_1} (1-\Norm(\idP)^{-ds})^e.$$
  Altogether we have for $s\in\C$ with $\Re(s)>1/d$:
  \begin{equation}\label{eq:tauber}
  f(s)=\prod_{\idp\in U}\left(1+\frac{m}{\Norm(\idp)^{ds}}\right) = g(s)g_0(s) \zeta_{k(\zeta_\ell)}(ds)^e\prod_{\idP\notin \tilde U_1} (1-\Norm(\idP)^{-ds})^e.
  \end{equation}

  All prime ideals in $\PP(k)$ not lying in $\tilde U_1$, lie over $\ell$, are ramified
  or have an inertia degree which is at least 2.  Since there are only
  finitely many prime ideals which are ramified or which lie over
  $\ell$, this Euler product is convergent for $s=1/d$ and
  has a value between 0 and 1.  Now we study the behavior of $\zeta_{k(\zeta_\ell)}(ds)^e$. Using the fact that the Dedekind $\zeta$-function has a simple pole at $s=1$ we choose $a:=1/d$ and $z:=e$  and we apply the
  Tauberian Theorem \ref{tauber} to get the desired result.
\end{proof}

The following sums will occur in our applications. We remark that the $B_i$ come from the fact that in a cyclic extension
of degree $\ell$ only prime ideals above $\ell$ and prime ideals congruent to 1 modulo $\ell$ can ramify.  Therefore by choosing $\ell_i=2$ we get the asymptotics without congruence condition. We see that the congruence condition changes the $\log$-part of the asymptotics. For the rest of the paper we denote by $\omega(\ida)$ the number of different prime ideal factors of $\ida$.

\begin{lemma}\label{logbound2}
  Let $k$ be a number field, $m_1,\ldots,m_r, d_1,\ldots,d_r\in\N$, and $\ell_1,\ldots,\ell_r\in\PP$. We define $n_{\ell_i}:=[k(\zeta_{\ell_i}):k]$,  $d:=\min(d_1,\ldots,d_r)$,
  $e:=\sum\limits_{i: d_i=d}\frac{m_i}{n_{\ell_i}}$, and
  \[B_i:=\{\idp\in\PP(k) \mid \Norm(\idp)\equiv 0,1 \bmod \ell_i\}.\] Then we get for a suitable constant $c>0$:
  $$\sum_{\Norm(\ida_1)^{d_1}\cdots\Norm(\ida_r)^{d_r}\leq x}
  m_1^{\omega(\ida_1)}\cdots m_r^{\omega(\ida_r)} \tilde c x^{1/d} \log(x)^{e-1},$$
    where the sum is over all squarefree ideals $\ida_i$ such that $\idp\mid \ida_i \Rightarrow \idp\in B_i$.
    %prime ideal factors of $\ida_i$ are in $B_i$.
\end{lemma}
\begin{proof}
  The corresponding Dirichlet series $F(s)$ can be written as a product
  $$F(s) = \prod_{i=1}^r F_i(s), \mbox{ where }
  F_i(s)=\prod_{\idp\in B_i} \left( 1+ \frac{m_i}{\Norm(\idp)^{d_i s}}\right).$$
  Using Lemma \ref{euler2} we see that $F_i$ are convergent for $\Re(s)>1/d_i$ with corresponding $z_i:=m_i/n_{\ell_i}$. The functions with $d_i>d$ are convergent Euler products at $s=1/d$ and therefore holomorphic with $F_i(1/d)\ne 0$. For the application of Theorem \ref{tauber} we define $a:=1/d$ and $z$ to be the sum of the $z_i$ and note that $z=e$.
   Altogether,
  $F(s) = \zeta_{k(\zeta_\ell)}(ds)^e G(s)$, where $G(s)$ is holomorphic at $s=1/d$ and $G(1/d)\ne 0$.
\end{proof}

\section{Class group bounds}
We denote by $\rk_\ell(\Cl_k):=\dim_{\F_\ell}(\Cl_k/\Cl_k^\ell)$ the $\ell$-rank of the class group $\Cl_k$ of the maximal order $\OO_k$ of $k$ and by $r_1$ the number of real places of $k$.
\begin{lemma} \label{lbound}
  Let $k$ be a number field, $\ell\in\PP$, and $S\subseteq \PP(k)$ be a finite set of prime ideals.  Then there
  are at most $\frac{\ell^s-1}{\ell-1}$ $C_{\ell}$-extensions of $k$ which
  are unramified outside $S$, where
  $$s=\begin{cases} \rk_\ell(\Cl\nolimits_k)+|S| +[k:\Q]& \ell>2.\\
                      \rk_\ell(\Cl\nolimits_k)+|S|+[k:\Q] +r_1& \ell=2. 
   \end{cases}$$
\end{lemma}
\begin{proof} 
	The idea of the proof is to find an ideal $\ida$ such that all wanted $C_\ell$-extensions are a subfield of the ray class field of $\ida$ and the at most $r_1$ real places of $k$. Certainly, we can define $\ida:=\prod_{\idp \in S} \idp^{a_\idp}$ for a suitable choice of the numbers $a_\idp\geq 1$. Denote the ray class group of $\ida$ and the infinite places by $A$. We would like to analyze the group $A/\Cl_k$ and 
  we use the notation of \cite{Lang}, pp.~123. Let $\idc$ be the cycle which is divisible by all real places and has finite part $\idc_0:=\ida$. We get \cite[p.~126]{Lang}:
 % $$A/\Cl\nolimits_k\cong (\OO_k/\fa)^\times / (U_k \cap k_\fa) \times \sum_{v \mbox{ real }}k_v^*/k_v^+\,,$$
  \[ A/\Cl\nolimits_k \cong (k(\idc)/k_{\idc})\;/\;(U/U_\idc).\]
  We remark that 
  \[k(\idc)/k_{\idc} \cong (\OO_k/\fa)^\times \times C_2^{r_1}\]
  in our notation, where the latter part is coming from the $r_1$ real places.
   
  We get that the $\ell$-rank of $A$ is at most $\rk_\ell(\OO_k/\fa)^\times+\rk_\ell(\Cl_k) + r_1$. This explains the $\rk_\ell(\Cl_k)$ and the $r_1$-part of the formula in the claim. We define $S_1:=\{\idp \in S \mid \ell\notin\idp \}$.
   Then by the Chinese remainder theorem we can write:
  $$(\OO_k/\fa)^\times \cong \prod_{\idp \in S} (\OO_k/\idp^{a_\idp})^\times = \prod_{\idp \in S_1} (\OO_k/\idp^{a_\idp})^\times \times
  \prod_{\idp \in S\setminus S_1} (\OO_k/\idp^{a_\idp})^\times.$$
  Therefore we have to study $(\OO_k/\idp^{a_\idp})^\times$ which is isomorphic to
  \begin{equation}\label{eq:rk_ray}
   (\OO_{k_\idp}/\idp^{a_\idp})^\times \cong (\OO_{k_\idp}/\idp)^\times \times (1+\idp)/(1+\idp^{a_\idp}),
  \end{equation}
  where we denote the unique prime ideal of $\OO_{k_\idp}$ by $\idp$ as well. The group $(\OO_{k_\idp}/\idp)^\times$ has size
  $\Norm(\idp)-1$ and $(1+\idp)/(1+\idp^{a_\idp})$ is a $p$-group, where $p\in \idp \cap \PP$.

  For $\idp\in S_1$ we have $p\ne\ell$ and  $\rk_\ell((\OO_k/\idp^{a_\idp})^\times) \leq 1$, since
  $(\OO_{k_\idp}/\idp)^\times$ is a cyclic group of order $\Norm(\idp)-1$ and the order of the second factor of \eqref{eq:rk_ray} is coprime to $\ell$.
  
  For $\idp\in S \setminus S_1$ we have $p=\ell$. The order of the first group is coprime to $\ell$.
  Using \cite[Theorem 5.21]{Nar} we get that $\rk_\ell((1+\idp)/(1+\idp^{a_\idp}))\leq [k_\idp:\Q_\ell]+1$. Since
  $[k:\Q] = \sum_{\idp\mid (\ell)} [k_\idp:\Q_\ell]$ we finish the proof.
\end{proof}
It was not necessary in the above proof to compute the numbers $a_\idp$.
We remark that for prime ideals $\idp$ in $S_1$  we can choose $a_{\idp}=1$, since bigger exponents do not change the $\ell$-rank of \eqref{eq:rk_ray}. The difficult part of the above lemma was to prove the upper bound for prime ideals lying over $\ell$. If not all prime ideals above $\ell$ are contained in $S$ we can improve the $[k:\Q]$-part in the above lemma. In case $S=S_1$ we can omit this part completely.

%In the preceding lemma we got an upper estimate for the number of $C_\ell$-extensions which are unramified outside $S$. 
In the following we also need a good upper bound for the number of
$C_\ell$-extensions which are exactly ramified in prime ideals in $S$.
%We remark that we do not carefully study wild ramification which leads to an error affecting only the constant of our
asymptotics.
  \begin{theorem}\label{bound_exact}
    Let $k$ be a number field, $\ell\in \PP$, and
    $S,T\subseteq \PP(k)$ be finite sets of prime ideals of $\OO_k$
    with empty intersection.  Then there exists a constant $c(k,\ell)$
    such that the number of $C_\ell$-extensions of $k$ which are ramified in
    all prime ideals of $S$ and are unramified outside of $S\cup T$ is
    bounded above by
    $$\ell^{c(k,\ell)+|T|} (\ell-1)^{|S|}.$$
  \end{theorem}
  \begin{proof}
Denote by $M$ the composite of all $C_\ell$-extensions of $k$ which are unramified outside $S\cup T$. Let $\Gal(M/k) \cong C_\ell^t$.
  	    Note that all the wanted $C_\ell$-extensions of the claim are subfields of $M/k$. In the following we would like to construct $C_\ell$-extensions $k_1,\ldots,k_t$ of $k$
    such that the composite $k_1\cdots k_t=M$ and the fields $k_i$ have nice properties.  We define $S_1:=\{\idp \in S \mid \ell\notin\idp \}$. Note that if $\idp$ is unramified in $M/k$ for one $\idp \in S_1$, then there are no $C_\ell$-extensions with the desired properties.

    For this we choose $\idp_1\in S_1$. Let $k_1\leq M$ be one
    $C_\ell$-extension which is ramified in $\idp_1$.  We claim that we can choose $k_2,\ldots,k_t$ in a way such
    that $k_2,\ldots, k_t$ are unramified in $\idp_1$. Certainly, we can
    find extensions $k_2,\ldots,k_t$ such that $M= k_1\cdots k_t$. Now
    assume that $k_i$ is ramified in $\idp_1$ and consider
    $L:=k_1k_i$. By Abhyankar`s lemma \cite[Cor. 4,p. 229]{Nar}, we
    see that $L/k_1$ is unramified over prime ideals lying above
    $\idp_1$ and therefore there exists a subfield (the inertia field)
    $k\leq \tilde k_i \leq L$ such that $\tilde{k}_i/k$ is unramified
    in $\idp_1$. Now we replace $k_i$ by $\tilde{k}_i$.
    Now we choose $\idp_2\in S_1$ such that $\idp_2$ is ramified in at
    least one of the fields $k_2,\ldots,k_t$. Applying the same
    procedure as above we can assume that $\idp_2$ is ramified in
    $k_2$ and unramified in $k_3,\ldots,k_t$. By replacing $k_1$
    with $\tilde k_1$ in $k_1k_2$ we can assume that $\tilde k_1$ is
    still ramified in $\idp_1$ and unramified in $\idp_2$. We continue this
    process as long we can find a prime ideal $\idp_i\in S_1$ which is ramified in at least one of the fields $k_i,\ldots,k_t$. 
    
    When the process is finished we have chosen $r\le |S_1|$ prime ideals $\idp_1,\ldots,\idp_r\in S_1$ and $C_\ell$-extensions   
    $k_1,\ldots,k_t$ with $M=k_1\cdots k_t$ and the property that for $1\leq i \leq r$ the prime ideal $\idp_i$ is ramified in $k_i$ and it is unramified in $k_j$ for $j\ne i$ and $1\leq j \leq t$. Note that it is possible that $r<|S_1|$. This happens for an example if $\idp_1\ne \idp \in S_1$ is only ramified in $k_1$, but not in $k_2,\ldots,k_t$.
    
     Now let $L/k$ be a $C_\ell$-extension which is contained in
    $k_2\cdots k_t$. Then $L/k$ is unramified in $\idp_1$ and
    therefore we need not count $L$. In the same way $L$ is not
    ramified in $\idp_i \;(1\leq i \leq r)$, if $L$ is contained in
    $\prod_{j\ne i} k_j$.  We would like to count the $C_\ell$-extensions
    of $M$ which are not contained in $\prod\limits_{j=1,j\ne i}^t k_j$
    $(1\leq i \leq r)$ since those are not ramified in $\idp_i$.  By
    inclusion-exclusion we get for this number:
    $$\frac{\ell^{t-r}}{\ell-1} \sum_{j=0}^r  (-1)^j \binom{r}{j} (\ell^{r-j}-1)
    =\frac{\ell^{t-r}}{\ell-1} \sum_{j=0}^r  (-1)^j \binom{r}{j} \ell^{r-j} - \frac{\ell^{t-r}}{\ell-1} \sum_{j=0}^r  (-1)^j \binom{r}{j}$$
    $$=\frac{\ell^{t-r}}{\ell-1} (\ell-1)^r - 0 = \frac{(\ell-1)^r\ell^{t-r}}{\ell-1}\leq (\ell-1)^{(|S_1|-1)} \ell^{c(k,\ell)+|T|}.$$
    For the last inequality we used that $k_{r+1}\cdots k_t$ is contained in the maximal elementary abelian $\ell$-extension unramified outside $T \cup S_0$, where $S_0:=\{\idp \in S \mid \ell\in\idp \}$. and we apply
    Lemma \ref{lbound}.
    %we can choose $c(k,\ell):=s-r-|T|$ and this number is bounded by $\rk_\ell(\Cl_k)+[k:\Q]+r_1+|S\setminus S|\leq \rk_\ell(\Cl_k) + 3[k:\Q]$.
  Note that we can choose $c(k,\ell)=\rk_\ell(\Cl_k)+[k:\Q]+|S_0| +r_1\le \rk_\ell(\Cl_k)+ 3[k:\Q]$. 
  \end{proof}
\section{Finding all solutions of certain embedding problems}\label{embed}

\begin{definition} \label{embeddef}
 Let $K/k$ be a Galois extension of number fields with group $H$ and
 $$1 \longrightarrow U \longrightarrow G 
     \stackrel{\kappa}{\longrightarrow} H \longrightarrow 1$$
 be an exact sequence of finite groups. Denote by $\bar{k}$ an algebraic closure of $k$ containing $K$ and by 
 $\varphi: \Gal(\bar{k}/k) \rightarrow H \cong \Gal(K/k)$ an epimorphism with
 kernel $\Gal(\bar{k}/K)$. This is called the {\it embedding problem given by
 $\kappa$ and $\varphi$}. An epimorphism 
 $\tilde{\varphi}: \Gal(\bar{k}/k) \rightarrow G$ 
 is called {\it (proper) solution} to the embedding problem, if 
 $\kappa\circ \tilde\varphi = \varphi$. The fixed field $L$ of 
 $\ker(\tilde\varphi)$ is called {\it solution field}. Furthermore, the
 embedding problem is called {\it central}, if $U$ is in the center of $G$.
\end{definition}

Suppose that we are given an extension $K/k$ with Galois group $H$ and
we want to find all fields $L$ containing $K$ such that $\Gal(L/k)\cong G$.
Since the embedding problems depend on $\varphi$ and $\kappa$ these can in
general not be obtained as solutions of a single embedding problem.

\begin{lemma} \label{numberembed}
 Let $K/k$ be an extension of number fields with Galois group $H$ and let $G$
 be a finite central extension of $H$. Then all fields $L/K$
 such that $\Gal(L/k)\cong G$ can be found as solutions of finitely many 
 embedding problems, whose number only depends on $G$ and $H$.
\end{lemma}

\begin{proof}
Let $\varphi$ be defined as in Definition \ref{embeddef}. Epimorphisms
$\kappa: G \rightarrow H$ with the same kernel only differ by automorphisms
of $H$. Furthermore, $G$ has only finitely many normal subgroups with quotient
isomorphic to $H$.
\end{proof}

Our goal in this section is to describe all solutions to an embedding problem if we assume that we know one solution.

\begin{definition}\label{def:subdirekt}
  Let $G_1, G_2, H$ be finite groups and 
  $\kappa_i: G_i \rightarrow H \;(i=1,2)$ epimorphisms. Then we define the subdirect product 
  $$G_1 \times_H G_2:=\{(g_1,g_2) \mid g_i\in G_i\;(i=1,2), \kappa_1(g_1)=\kappa_2(g_2)\}$$
with componentwise operation.
\end{definition}
Note that this construction is also called fiber product or pullback. The order of this group
is $|G_1||G_2|/|H|$. Suppose that we have an embedding problem as in Definition \ref{embeddef} and
we have two different solutions $L_1$ and $L_2$ such that $L_1 \cap L_2 =K$. Then $\Gal(L_1L_2/k) =
G\times_H G$.

\begin{lemma}\label{satz:subd}
  Let $G$ be a group extension as in  Definition \ref{embeddef}. 
  We define
  $$\Psi: G \rightarrow \Aut(U), g \mapsto (u \mapsto u^{(g^{-1})} = gug^{-1}).$$ 
  Then
  $$G \times_H G \cong U \rtimes G,$$
  where the semidirect product is defined via $\Psi$.
\end{lemma}
\begin{proof}
  We define $\Phi: U \rtimes G \rightarrow G \times_H G , (u,g)
  \mapsto (g,ug)$.  Since $\kappa(u)=1$ for $u\in U$ we have $(g,ug)\in G \times_H G$.
  Furthermore, $\Phi$ is a homomorphism since 
  $$\Phi(u_1,g_1) \Phi(u_2,g_2) = (g_1,u_1g_1)(g_2,u_2g_2)=(g_1g_2,u_1g_1u_2g_2)$$
  and $$\Phi((u_1,g_1)(u_2,g_2))
  =\Phi((u_1u_2^{(g_1^{-1})},g_1 g_2)) = (g_1 g_2, u_1u_2^{(g_1^{-1})}g_1g_2)
  =(g_1g_2,u_1g_1u_2g_2).$$
  Finally, $\Phi$ is injective and then, also surjective since both groups
  have the same finite order.
\end{proof}

In the following we concentrate on embedding problems with abelian kernel $A$:
\begin{equation}\label{einbettab}
1 \longrightarrow A \longrightarrow G 
     \stackrel{\kappa}{\longrightarrow} H \longrightarrow 1
   \end{equation}
   We interpret $A$ as an $H$--module via
   $\Psi: H \rightarrow \Aut(A): h \mapsto (a \mapsto
   a^{\kappa^{-1}(h)})$.  We can take an arbitrary preimage of $h$
   since $A$ is abelian. Note that this not well defined for non
   abelian normal subgroups. For given $A$, $H$, and $\Psi:H \rightarrow \Aut(A)$ we
   construct a semidirect product as above.

   The following theorem is the main result of this section. It gives a relation between
   split and non-split embedding problems with abelian kernel.
\begin{theorem} \label{satz:subd2}
  
  Let $G$ be a group extension as in \eqref{einbettab} and
  $\Psi: H \rightarrow \Aut(A), h \mapsto (a \mapsto a^{(\kappa^{-1}(h))^{-1}})$.
  Let $A \rtimes H$ be the corresponding semidirect product. Then:
  $$G \times_H G \cong  G \times_H (A \rtimes H).$$
  Furthermore,
  $A\rtimes H \cong (G\times_H G)/\{(a,a)\mid a \in A\}$.
\end{theorem}
\begin{proof}
  The homomorphism $\Psi$ can be canonically extended to $G$.
  Using Lemma \ref{satz:subd} we get
  $G\times_H G \cong A \rtimes G$.  Furthermore,
  $\Phi: A \rtimes G \rightarrow A \rtimes H, (a,g) \mapsto (a,\kappa(g))$ is a surjective homomorphism since the automorphism corresponding to $h \in H$ is independent of the preimage under $\kappa$. The kernel equals $\{(1,a)\mid a \in  A\}$.
  Using the isomorphism of Lemma \ref{satz:subd} this kernel will be mapped to 
  $\{(a,a) \mid a \in A\}\normal G \times_H G$.
  We get that
  $$A\rtimes G  \rightarrow G \times_H (A \rtimes H), (a,g) \mapsto ( g,
  (a,\kappa(g)))$$ is an isomorphism.
\end{proof}

The application of this theorem corresponds to the following field diagram:
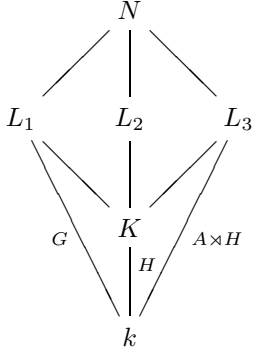
\begin{minipage}{.31\linewidth}\vspace{0pt}
$$\begin{diagram}
\node[2]{N} \arrow[1]{se,-} \arrow[1]{s,-} \arrow{sw,-}\\
\node[1]{L_1} \arrow{se,-} \arrow{sse,r,-}{G}\node{L_2} \arrow[1]{s,-} \node{L_3} \arrow{ssw,r,-}{A\rtimes H}\arrow[1]{sw,-}\\
\node[2]{K} \arrow{s,r,-,1}{H} \\
\node[2]{k} 
\end{diagram}$$
\end{minipage}
\begin{minipage}{.68\linewidth}\vspace{0pt}
We assume that $\Gal(L_1/k)\cong\Gal(L_2/k)\cong G$, $L_1\cap L_2=K$, and $N=L_1L_2$.
Then there exists a field $K\leq L_3\leq N$ with $\Gal(L_3/k)= A\rtimes H$. Similar, if we know
fields $L_1,L_3$ with the above Galois groups (and $L_1L_3=N, L_1\cap L_3=K$) then the theorem gives us $L_2$ with Galois group $G$.
Depending on the embedding problem there are more fields $K< L < N$ with Galois group $G$.
This number is a group theoretic invariant which can be computed in concrete situations.
\end{minipage}
In this paper we would like to study nilpotent groups. Since non-trivial nilpotent groups have non-trivial center, we
can reduce to the case that $A=C_\ell$ is a cyclic group of prime order in the center. We get:
\begin{equation}\label{einbett}
1 \longrightarrow C_\ell \longrightarrow G 
     \stackrel{\kappa}{\longrightarrow} H \longrightarrow 1 \;\mbox{ central.}
\end{equation}

The easiest case is when a primitive $\ell$-th root of unity $\zeta_\ell$ is already contained in
$k$. Then the embedding problem corresponding to \eqref{einbett} is a so called Brauer embedding problem,
e.g. see \cite[chapter IV.7]{MM}.

Since $\zeta_\ell\in k$ we then have solution fields $L/K$ of the form 
$L=K(\sqrt[\ell]{\alpha})$ for suitable $\alpha\in K$ (Kummer theory). There are other nice properties
of Brauer embedding problems, e.g., a local-global principle is valid. For us the following lemma is important.
The proof can be found in  \cite[chapter IV.7]{MM}.
\begin{lemma} \label{brauerlem}
  Let $L=K(\sqrt[\ell]{\alpha})$ be a solution field of a  Brauer embedding problem. Then all solution
  fields are given by
  $K(\sqrt[\ell]{a \alpha})$ with suitable  $0\ne a \in k$.
\end{lemma}                     

This lemma was the starting point of the approach in \cite{KlMa2}.
Via Kummer theory a similar correspondence was shown for fields $k$ not containing $\zeta_\ell$.
We show this correspondence using the group theoretic approach we have just developed.

\begin{theorem}\label{satzzen}
  Let $G$ be a central extension as in  \eqref{einbett} and
  $\tilde{G} = C_\ell \times G \cong G \times_H G$. 
  Then for
  $D:=\{(a,b)\in C_\ell \times G\mid \kappa(b)=1\}\cong C_\ell\times C_\ell$ we get:
  \begin{enumerate}
  \item $D \leq \Zen(\tilde{G}) = C_\ell \times \Zen(G)$.
  \item $D$ has $\ell+1$ subgroups
    $U_1,\ldots,U_{\ell+1}$ of order $\ell$. All these subgroups are normal in 
    $\tilde{G}$.
  \item All quotients $\tilde{G}/U_i\;(1\leq i \leq \ell+1)$ except one
    are isomorphic to $G$. The remaining quotient is isomorphic to
    $C_\ell\times H$.
  \end{enumerate}
\end{theorem}
\begin{proof}
  The isomorphy $C_\ell \times G \cong (C_\ell \times H) \times_H G \cong G \times_H G$ is Theorem \ref{satz:subd2}.
  The first two assertions are elementary. Let $a$ be a generator of the kernel of $\kappa$.
  We remark that we interpret the first factor of the direct product as the subgroup in the kernel of
  $\kappa$.
  Then we describe the subgroups $U_i \normal C_\ell\times G$ as follows
  $$U_i:=\langle(a,a^i)\rangle \;(\mbox{ for }1\leq i \leq \ell)\mbox{ and }
  U_{\ell+1}:=\langle (1,a)\rangle.$$
  Clearly, we get $\tilde{G}/U_{\ell+1} \cong C_\ell \times H$.
  We define the following images under the isomorphism $\Phi$ from the proof of Lemma \ref{satz:subd}:
  $$\tilde{U}_i= \Phi(U_i)=\langle(a^i,a^{i+1})\rangle \mbox{ for }1\leq i  \leq \ell.$$
  Using the isomorphism we get
  $(C_\ell \times G)/U_i \cong (G \times_H G)/\tilde{U}_i$.  Let
  $(g,\tilde{g})\in G \times_H G$. Since
  $\kappa(g)=\kappa(\tilde{g})$ we get $\tilde{g}=ga^j$ for some
  $j\in \{0,\ldots,\ell-1\}$.\\   We choose 
  $l\equiv -j \bmod \ell$ and get:
 $$(g,ga^j)(a^i,a^{i+1})^l 
%= (g,ga^j)(a^{il},a^{(i+1)l})
 = (ga^{il},ga^{j+(i+1)l}) =
 (ga^{il},ga^{-l+(i+1)l})=(ga^{il},ga^{il}).$$ Therefore we find in
 the quotient $(G \times_H G)/\tilde{U}_i$ representatives for each
 element, whose components coincide. Therefore this quotient is isomorphic to $G$.
\end{proof}

This theorem shows that we can describe all solutions of our embedding problem  \eqref{einbett}, if we
know one solution. Let $L/K$ be a solution field and $M/k$ be a
$C_\ell$-extension with $[ML:L]=\ell$. Using Theorem \ref{satzzen} we get  $\ell-1$
new solution fields of our embedding problem.
Two cyclic $C_\ell$-extensions $M_1,M_2$ of $k$ parameterize the same solution fields, if $LM_1=LM_2$.
This condition defines on the set of $C_\ell$-extensions of $k$ an equivalence relation. For technical
reasons the trivial extension is contained in the next definition.
$${\mathcal M}_K:=\{M/k \mid \Gal(M/k) \leq C_\ell\}/\sim,$$
where $M_1\sim M_2$ if and only if $LM_1=LM_2$. We see that $\sim$ is
only depending on the maximal $\ell$-elementary abelian subextension
of $L/k$. For $M_1\leq L$ we get that $M_1$ is in the trivial
class. Let $M_1,M_2$ be not in the trivial class. Then $M_1\sim M_2$
if and only if $M_1M_2\cap L \ne k$.

Since $L/k$ contains only finitely many $C_\ell$-extensions, we get:
\begin{theorem}\label{satzzen2}
  Let $K/k$ be a normal extension of number fields with $\Gal(K/k)=H$. Denote by
  $L/K$ a solution field of our embedding problem \eqref{einbett}.
  We define 
  $$\psi:  \{\tilde{L}/K \mid \tilde{L} \mbox{ solution field}\} \rightarrow
  {\mathcal M}_K, \tilde{L} \mapsto [M],$$
  where $M/k$ is a cyclic extension with $LM=L\tilde{L}$. Then we get:
  \begin{enumerate}
  \item $\psi$ is a mapping.
  \item The trivial class  $[M]$ with $M\leq L$ has exactly one preimage.
    All other classes have exactly $\ell-1$ preimages.
  \item Denote by $r$ the $\ell$--rank of the maximal abelian quotient of $G$.
    Then the trivial class contains exactly $\frac{\ell^{r}-1}{\ell-1}+1$ elements.
    All other classes have exactly $\ell^{r}$ elements.
  \end{enumerate}
\end{theorem}
\begin{proof}
  Using Theorem \ref{satzzen} and the preceding discussion every solution $\tilde{L}$ corresponds to exactly one class in ${\mathcal M}_K$. This shows that $\psi$ is a mapping.
  Furthermore, only the special solution $L$ corresponds to the trivial class.
  Using Theorem \ref{satzzen} we see that $\ell-1$ solutions correspond to a non-trivial class.
  The  extension $K/k$ has exactly $\frac{\ell^{r}-1}{\ell-1}$ subfields with Galois group $C_\ell$ over $k$.
  Together with the trivial extension $k$ these fields are the trivial class.
  For a non-trivial class with representative $M$  we get that
  $LM$ has $\frac{\ell^{r+1}-1}{\ell-1}$ intermediate fields with Galois group $C_\ell$ over $k$.
  Exactly $\frac{\ell^{r}-1}{\ell-1}$ of those fields are contained in $L$. Therefore we get
  $$\frac{\ell^{r+1}-1}{\ell-1} - \frac{\ell^{r}-1}{\ell-1} 
  = \frac{\ell^{r+1}-\ell^{r}}{\ell-1} = \ell^{r}$$
  fields in the equivalence class of $M$.
\end{proof}

Using Theorem \ref{satzzen2} we parameterize all solutions of a central embedding problem with cyclic kernel of order $\ell$ via cyclic extensions of $k$. In the following we have to solve the following two problems:
\begin{enumerate}
\item How do we choose the special solution $L$?
\item How can we compute the discriminant $d_{\tilde L/k}$ from $L$ and $M$? 
\end{enumerate}

\section{Nilpotent groups} \label{nilp}

\subsection{Reduction to $\ell$-groups}
Let $G\leq S_n$ be a transitive nilpotent group. It is well known that
$G$ is a direct product of its Sylow $\ell$-subgroups.  If we consider
$G$ as a permutation group, we can prove (e.g. see \cite[Lemma 3.1]{KlWa}) that $G$ is a natural direct product of its Sylow $\ell$-subgroups, where the action as permutation groups is given by the action on the blocks.
\begin{lemma}\label{nil_a}
 Let $G_1\leq S_{n_1} \text{ and } G_2\leq S_{n_2}$ be transitive nilpotent groups which are of coprime order and assume $n_1,n_2>1$. Define the natural direct product $G =G_1\times G_2\leq S_{n_1n_2}$. Then $a(G) = \max(a(G_1)/n_2,a(G_2)/n_1)$ and the two numbers in the maximum
  are different.
\end{lemma}
\begin{proof}
  The equality is proved in \cite[Lemma 4.1]{Ma4}. Suppose that the two numbers are equal. Using
  $a(G_i)^{-1}=\ind(G_i)$ we get $n_1a(G_1) = n_2a(G_2)$ which is equivalent to
  $$n_1\ind(G_2) = n_2\ind(G_1).$$
  Since $\gcd(n_1,n_2)=1$ we get that $n_1\mid \ind(G_1)$, but by definition we have $1\leq \ind(G_1)<n_1$,
  since $G_1$ is not the trivial group. This is a contradiction.
\end{proof}
Putting these results together we get the following proposition.
\begin{proposition}\label{thm:nilp}
	A transitive nilpotent permutation group $G\leq S_n$ is permutation isomorphic to the natural direct product of transitive permutation groups $G_{\ell}\leq S_{n_{\ell}}$,
	$$G \simeq \prod_{\ell} G_{\ell} \mbox{ with } n=\prod_{\ell} n_{\ell},$$
	where the $G_{\ell}$ are isomorphic to the Sylow $\ell$-subgroups of $G$ and $n_{\ell}$ is the maximal $\ell$-power dividing $n$. Furthermore, we have that
	$$a(G) = \max(na(G_{\ell})/n_{\ell} \mid \ell\in  \PP).$$
	The maximum number $na(G_{\ell})/n_{\ell}$ is attained exactly once.
\end{proposition}
 \begin{proof}
   This is the content of \cite[Lemma 3.1]{KlWa} and the extension of Lemma \ref{nil_a} to more than two factors.
 \end{proof}
 \begin{corollary}\label{crit_prime}
   Let $G\leq S_n$ be a transitive nilpotent group. Then all elements of $G$ of minimal index have the same order. 
 \end{corollary}
 \begin{proof}
   Using Proposition \ref{thm:nilp} we see that all elements of minimal index lie in the same Sylow $\ell$-subgroup for
   some $\ell\in\PP$. An element of minimal index is certainly of prime order, since otherwise a suitable power has more
   orbits. Therefore all these elements have order $\ell$.
 \end{proof}
 
\subsection{The general case}
Let $G$ be a finite nilpotent group and denote by $E$ the trivial group.
We choose a refinement of an upper central series of $G$ in a way that
each quotient is of prime order. We remark that the refinement is not unique and
for the following parts there might be good or bad choices to do the refinement.

Therefore we find groups
\begin{equation}\label{choice_Gi}
E=G_r < G_{r-1} < \ldots < G_{i} < G_{i-1} < \ldots < G_{1} < G_0 =G.
\end{equation}
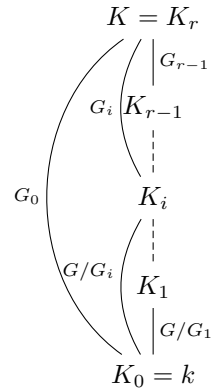
\begin{minipage}{.25\linewidth}\vspace{10pt}
	\begin{tikzcd}
	K=K_r
	\arrow[-]{d}{G_{r-1}}   \arrow[bend right=55,-]{dddd}[left]{G_0} \arrow[bend right,-]{dd}[swap]{G_i}\\
	K_{r-1}
	\arrow[-,dashed]{d}\\
	K_i \arrow[-,dashed]{d}\arrow[bend right,-]{dd}[above left]{G/G_i} \\
	K_{1} \arrow[-]{d}{G/G_{1}} \\
	K_0=k\\
	\end{tikzcd}
\end{minipage}
\begin{minipage}{.74\linewidth}\vspace{0pt}%\hspace{-10pt}
We remark that all $G_i$ are normal subgroups of $G$, all quotients $G_{i-1}/G_{i}$
are of prime order, and $G_{i-1}/G_i$ lies in the center of $G/G_i$ for all $1\leq i \leq r$.
Suppose we have an extension $K/k$ with Galois group $G$. Then we get
a tower of fields $k=K_0<\ldots< K_r=K$ such that
$\Gal(K/K_i)=G_i$.  Let $\idp$ be a prime ideal of $\OO_k$ and $\idP$ be a prime ideal
of $\OO_K$ lying above $\idp$. Denote by $G_\idP$ the inertia group of $\idP$.
We remark that the inertia groups of all
prime ideals lying over $\idp$ are conjugate.

Now we want to consider the situation that $\idp$ firstly becomes ramified in $K_i$, i.e.
$\idp$ is ramified in $K_i$ but not in $K_{i-1}$. We use the following easy lemma.
\end{minipage}

\begin{lemma}
The prime ideal $\idp$ is unramified in $K_i$ if and only if $G_\idP\leq G_i$.
\end{lemma}
\begin{proof}
  By Hilbert`s ramification theory, the fixed field $L$ under $G_{\idP}$ is the largest subfield
  of $K$, for which $L \cap \idP$ is unramified over $\idp$. Since $K_i/k$ is normal, $\idp$ is unramified
  when one of the prime ideals lying over $\idp$ is unramified. Therefore $\idp$ is unramified in $K_i$
  if and only if $K_i$ is a subfield of $L$ which is equivalent to $G_\idP\leq G_i$.
\end{proof}
Now we see that $\idp$ firstly becomes ramified in $K_i$ if and only if
$G_\idP \leq G_{i-1}$ and $G_{\idP} \not\subseteq G_i$. This condition becomes much easier in the tamely ramified situation.
\begin{lemma}
  Let $\idp$ be a prime ideal of $\OO_k$ which is at most tamely ramified in $K$. Denote by $\sigma$
  a generator of the cyclic group $G_{\idP}$. Then $i$ is minimal such that
$\idp$ ramifies in $K_i$  
   if and only if
  $\sigma \in G_{i-1} \setminus G_i$.
\end{lemma}
\begin{proof}
  If $\idp$ is not wildly ramified, we know that $G_{\idP}=\langle \sigma \rangle$ is cyclic.
  Therefore $G_{\idP}\leq G_i$ is equivalent to $\sigma\in G_i$.
\end{proof}
%In the following we assume that $\idp$ is tamely ramified and we denote by $\sigma$ a
%generator of the cyclic inertia group of $\idP$. Using Lemma \ref{lemind} we know that
%$$v_{\idp}(d_{K/k}) = \ind(\sigma).$$

If the
prime ideal $\idp$ is at most tamely ramified, we can use the index to describe the $\idp$-part of the discriminant (see also \cite[Section 7]{Ma4}). 
\begin{lemma}\label{lemind}
	Let $K/k$ be an extension of number fields with $\Gal(K/k)=G\leq S_n$ and $\idp\in\PP(k)$ be a tamely ramified prime ideal. Denote by $\sigma$ a generator of the (cyclic) inertia group (of the normal closure of $K/k$) and consider $\sigma$ as an element of the permutation group $G$. Then $v_\idp(\disc_{K/k}) = \ind(\sigma)$.
\end{lemma}
\begin{proof}
	The proof follows from Dedekind's different theorem and results about the decomposition of prime ideals in intermediate fields, see \cite[Section 6.3]{Koc}.
\end{proof}

Let
\begin{equation}\label{eq:ai}
  A_i:=G_{i-1} \setminus G_{i}\mbox{ for }1\leq i \leq r, a_i:=\min_{g\in  A_i}(\ind(g)) \mbox{ and } \ell_i:=\frac{|G_{i-1}|}{|G_i|}.
\end{equation}

Since the inertia generator must be one of the elements of $A_i$ we
get that $v_{\idp}(d_{K/k}) \geq a_i$.  Define $\ida_i\subseteq \OO_k$
to be the product of the ideals which become
firstly ramified in $K_i$.  In this way we associate to $K/k$ the
tuple of ideals $(\ida_1,\ldots,\ida_r)$. Note that all $\ida_i$ are squarefree and pairwise coprime. Furthermore all prime ideals $\idp$ dividing $\ida_i$ are contained in the set $B_i$ of Lemma \ref{logbound2}.
For the discriminant we get the estimate:
$$\Norm(d_{K/k}) \geq \Norm(\ida_1)^{a_1}\cdots \Norm(\ida_r)^{a_r}.$$

In order to obtain our wanted upper bound we have to compute how many
fields can be associated to the same tuple of ideals. We remark that
from the definition of the tuple we have that the ideals are pairwise
coprime. Denote by $I_k$ the group of ideals of $\OO_k$ and define the
following map using the above notation: 
$$\Phi: \{K/k \mid \Gal(K/k) = G\} \rightarrow I_k^r.$$ 
As before $\omega(\ida)$ denotes the number of different prime ideal factors of $\ida$.
\begin{theorem} \label{count}
  Let $k$ be a number field and $G\leq S_n$ be a transitive nilpotent group. We choose the $G_i$ given in \eqref{choice_Gi}.  Define
  $$b_i:=\omega(\prod_{j=1}^{i-1} \ida_j) + c(k,\ell_i),$$
  where $c(k,\ell_i)$ defined in Theorem \ref{bound_exact} only depends on $[k:\Q]$ and $\rk_{\ell_i}(\Cl_k)$.
  Then, all fibers of the map $\Phi$
  described above are finite of size at most
  $$\ell_1^{b_1}\cdots \ell_r^{b_r} (\ell_1-1)^{\omega(\ida_1)} \cdots (\ell_r-1)^{\omega(\ida_r)}.$$
\end{theorem}

\begin{proof}
  In Theorem \ref{satzzen2} we have proved that all solutions of a
  central embedding problem with kernel $C_{\ell}$ can be associated to
  cyclic extensions of order $\ell$ of $k$.  In order to get an upper
  bound we have to count these cyclic extensions which are at most
  ramified in the ideals $\ida_i$. Suppose we are in the $i$-th step.
  Then we have to compute the number of $C_{\ell_i}$-extensions which are
  at most ramified in the ideals dividing $\ida_1\cdots\ida_i$ and
  which are ramified in all ideals dividing $\ida_i$. Define
  $S:=\{\idp \in \PP(k) : \idp \mid \ida_1\cdots \ida_{i-1}\}$ and
  $T:=\{ \idp \in \PP(k) : \idp \mid \ida_i\}$ and apply Theorem \ref{bound_exact}. Then we get $\ell_i^{b_i}(\ell_i-1)^{\omega(\ida_i)}$ for
  the $i$-th step.
 \end{proof}

For the final step note that $A_i,a_i$, and $\ell_i$ are defined in \eqref{eq:ai}.
\begin{theorem} \label{count2}
  Let $G\leq S_n$ be a transitive nilpotent group and $k$ be a number field. Then there exist constants $d(k,G)\geq b(k,G), c(k,G)>0$ such
  that for all $x> 1$:
  $$Z(k,G;x) \leq c(k,G) x^{a(G)} \log(x)^{d(k,G)-1}.$$
\end{theorem}
\begin{proof}
  Using Theorem \ref{count} and the discriminant estimate we get:
  $$Z(k,G;x) \leq \sum_{\Norm(\ida_1)^{a_1}\cdots\Norm(\ida_r)^{a_r}\leq x}
  \ell_1^{b_1}\cdots \ell_r^{b_r} (\ell_1-1)^{\omega(\ida_1)}\cdots (\ell_r-1)^{\omega(\ida_r)}$$
  $$\leq   \prod_{i=1}^r \ell_i^{c(k,\ell_i)}  \sum_{\Norm(\ida_1)^{a_1}\cdots\Norm(\ida_r)^{a_r}\leq x}
  m_1^{\omega(\ida_1)} \cdots m_r^{\omega(\ida_r)},$$
  where  $m_r=(\ell_r-1), m_{r-1}=\ell_r(\ell_{r-1}-1),\ldots,m_1=\ell_r\cdots \ell_2(\ell_1-1)$.
  Using Lemma \ref{logbound2} we get the desired bound since $a(G)^{-1}$ is 
  the minimal index which occurs in the group $G$. Note that $\idp\mid \ida_i$ implies that $\idp\in B_i$ in the notation of Lemma \ref{logbound2}.
\end{proof}
Now we compute $d(k,G)$ which furthermore depends on our chosen groups $G_i$ in \eqref{choice_Gi}.
Let $\ell$ be the order of elements of smallest index, see Corollary \ref{crit_prime}.
Using $a:=\min(a_1,\ldots,a_r)$ and $n_\ell:=[k(\zeta_\ell):k]$ we define  
\begin{equation}\label{eq:upp}
d(G) := \sum_{i: a_i=a}m_i \text{ and }
d(k,G) :=d(G)/n_\ell.
\end{equation}

\begin{remark} \label{rem:d}
  In the situation of Theorem \ref{count2} we have for all valid choices in \eqref{choice_Gi}:
  $$\#\{g\in G\mid \ind(g)=\ind(G)\} \leq d(G)\leq |G|-1.$$
\end{remark}
\begin{proof}
  Using \eqref{choice_Gi} and \eqref{eq:ai} we see that $$G\setminus \{1\} = A_1\dot\cup\ldots\dot\cup A_r \mbox{ and }
  m_i=|A_i|=(\ell_i-1)\prod_{j=i+1}^r \ell_j.$$
  Inserting this in \eqref{eq:upp} we get the upper bound
  $$d(G) = \sum_{i: a_i=\ind(G)} m_i = \sum_{i: a_i=\ind(G)} |A_i|\leq |G|-1,$$
  since $1$ is not contained in any $A_i$.
\end{proof}
We remark that the constant $b(k,G)$ of Malle is trivially bounded above
by the number of elements with minimal index.

As already mentioned the proven upper bound for $d(k,G)$ is depending on our choice in \eqref{choice_Gi}. In the following we would like to find an optimal choice. As before let $\ell$ be the order of an element of smallest index in $G$.
 In a first step we choose our series of groups $E=G_r<G_{r-1}<\ldots<G_0=G$ in a way
such that there exists an $i$ with $G_i$ is an $\ell$-group and $G/G_i$ is coprime to $\ell$. This means
that $\Gal(K_r/K_i)$ is an $\ell$-group and the minimal index occurs for $j$ with $j\geq i$. From
this we see that we can prove the same upper bound for $d(k,G)$ as we can do for the
corresponding Sylow $\ell$-subgroup. The $a(G)$-constant is certainly independent of the choice of \eqref{choice_Gi}.

In order to study the values $d(k,G)$ we restrict to $\ell$-groups $G$. The following examples show that different choices lead to different results. We use equation \eqref{eq:ai}.
\begin{example}
  Let $G= \langle a,b\mid a^4 = b^2 =1,ab=ba\rangle \cong C_4 \times C_2\leq S_8$. We define $E=G_3<G_2= \langle a^2 \rangle<G_1 =\langle a^2,b\rangle<G_0=G$. Then $A_i:=G_{i-1}\setminus G_i$ for $1\leq i \leq 3$ and $a_i:=\min_{g\in A_i}(\ind(g))$.
  All elements in $A_1$ have order 4 and therefore $\ind(g)=6$ for all $g\in A_1$. All elements in $A_2$ and $A_3$ have
  order 2 and their index is 4. Now $d(k,G)=d(G) = |A_2|+|A_3| = 2+1=3=b(k,G)$ for all number fields $k$.

  If we choose $E< \langle a^2 \rangle < \langle a \rangle <G$, we get that the corresponding set $A_1$ contains elements
  of order 4 and of order 2 and therefore $a_1=4$. Furthermore, we get $a_2=6$ and $a_3=4$. Therefore
  $d(k,G)=d(G) = |A_1|+|A_3| = 4+1=5.$
\end{example}
For an abelian $\ell$-group it is clear how to minimize $d(G)$ (and $d(k,G)$). A trivial lower bound for $d(G)$ is the number
of elements of order $\ell$ since the minimal index is always attained at an element of prime order. In a
transitive abelian group the index of an element is only depending on the order of the element.
\begin{lemma}\label{ab}
  Let $G$ be an abelian $\ell$-group of order $\ell^r$ and $\ell$-rank $s$. If we choose the central series
  in a way such that $G_{r-s} \cong C_\ell^s$ then we get that $d(G) = \ell^s-1$ and $d(k,G) = b(k,G)$.
\end{lemma}
\begin{proof}
  Denote by $a:=\ind(G)$.  The sets $A_1,\ldots,A_{r-s}$ do not
  contain elements of order $\ell$ and therefore $a_i>a$ for
  $1\leq i \le r-s$.  The sets $A_{r-s+1},\ldots, A_r$ only contain
  elements of order $\ell$ and we get that
  $\sum_{i=r-s+1}^r |A_i| = |G_{r-s}|-1=\ell^s-1$.  Now we have
  $d(k,G):=d(G)/n_\ell$, where $n_\ell:=[k(\zeta_\ell):k]$. This shows
  $d(k,G)=b(k,G)$.
\end{proof}
We remark that Theorem \ref{count2} proves the correct upper bound including $\log$-factor for abelian groups $G$.
We have seen that our constant $d(G)$ is bounded below by the number of elements of minimal index. For non-abelian
groups Malle conjectures a dependency on conjugacy classes of minimal index. Usually a conjugacy class contains more
than one element in the non-abelian case. Therefore we cannot expect to get the expected upper bound for non-abelian
nilpotent groups. Sometimes we are lucky.

Let us study the non-abelian groups of order 8. Both groups $G$ have a
center $Z(G)$ of order 2 and therefore the upper central series is
given by $E< Z(G) < G$, where $G/Z(G)$ is a group of order 4 and we
have a choice for the refinement of the upper central
series. Since the center contains only one element of order 2, we get
$A_3=Z(G) \setminus \{1\}$ and therefore $a_3=4$ in all cases. In the case of the quaternion group $Q_8$ all
elements not contained in the center have order 4 and index
6. Independent of the refinement we get $a_1=a_2=6$ for $G=Q_8$. Let
us consider the dihedral group on 8 points. This group has exactly one
cyclic subgroup of order $4$. If we choose this group for $G_1$, we get that $A_2=G_1\setminus Z(G)$ contains two elements of order 4. The set $A_1=G\setminus G_1$ contains the remaining elements, all of them have order 2. Altogether, we get $a_1=4$ and $a_2=6$. We proved the correctness of the following
examples. If we choose a Klein four group for $G_1$, then we get
$a_1=4=a_2$. Here we have the problem that $A_1$ contains elements of
order 4 and of order 2, which gives worse exponents for our bound.

\begin{example}
  \begin{enumerate}
  \item Let $G=Q_8$. Then we get
    $Z(k,G;x) \leq O(x^{a(G)})$, which is as conjectured. Note that
    $a_1=a_2=6$ and $a_3=4$.
  \item  Let $G=D_4\leq S_8$. Then we get
    $Z(k,G;x) \leq O(x^{a(G)}\log(x)^4)$. Malle predicts exponent 2 instead of 4. Note that
    $a_1=a_3=4$ and $a_2=6$. The difference between 4 and 2 can be explained by the fact that there are 5 elements of order 2 which lie in 3 conjugacy classes.
  \end{enumerate}
  For $Q_8$ we get a sharp bound. For $D_4\leq S_8$ the problem is that
  we count many $C_2\times C_2$-extensions, for which the embedding problem into
  $D_4$ is not solvable.
\end{example}
Note that the bound for $D_4$ we got in the last example is optimal for our method. We get that $|A_1|+|A_3|=4+1=5$. Note that $D_4$ contains 5 elements of order 2, so we cannot do better with our method. Other choices of refinements would lead to a log-power of $7-1$.

When we consider the dihedral group $D_4\leq S_4$ as permutation group on 4 points, the non-trivial central element
has index 2 and therefore $a_3=2$. If we choose the subgroup of order 4 in a way that it contains the central
element and the transpositions, we get $a_2=1$ and $a_1=2$. Therefore we get $d(G)=2$ and we get
$Z(k,D_4;x) = O(x\log(x))$ which is one $\log$-factor more than expected. When we choose other refinements,
we get $a_1=1$ and $d(G)=4$. We remark that it is known \cite{Bai, CoDiOl2} that $Z(k,D_4;x) \sim c x$. 

What is special about $G=Q_8$? The important property is that all elements of minimal index lie in the center of the group. We remark that this is trivially true for abelian groups. Elements in the center of  a transitive group are always given in regular representation and therefore their index is only depending on their order. Similar to the abelian case we can prove.
\begin{lemma}\label{nil}
	Let $G\leq S_n$ be a transitive nilpotent group such that all elements of minimal index lie in the center of $G$. Then the number of those elements equals $\ell^s-1$ for a prime number $\ell$ and $s\in\N$. If we choose the central series
	in a way such that one $G_{i} \cong C_\ell^s$ then we get that $d(G) = \ell^s-1$ and $d(k,G) = b(k,G)$.
\end{lemma}
\begin{proof}
	Note that all elements of minimal index have prime order $\ell$. Since all these elements lie in the center we can choose one $G_i$ to exactly contain those elements. We proceed as in the proof of Lemma \ref{ab}. For the final equality note that the size of a conjugacy class of an element in the center is 1.	
\end{proof}
We remark that a converse of Lemma 5.13 is also true. If at least one element of minimal order does not lie in the center, then the conjugacy class of this element contains more than one element. Since $d(G)=d(k(\zeta_\ell),G)$ is bounded below by the number of elements of minimal index, we see that $b(k(\zeta_\ell),G)<d(k(\zeta_\ell),G)$, since $b(k(\zeta_\ell),G)$ equals the number of conjugacy classes of minimal index.

Let us define the generalized quaternion groups of order $4n$ for $n=2^k$ for all $k\ge 1$:
     \[ Q_{4n}:={\displaystyle \langle x,y\mid x^{2n}=y^{4}=1,x^{n}=y^{2},y^{-1}xy=x^{-1}\rangle }.\]
Note that $Q_{4\cdot 2}$ is the usual quaternion group and that those groups have only one transitive representation, which is the regular one acting on $4n$ points. These groups have only one element of order 2, which therefore lies in the center. Lemma \ref{nil} yields for $n=2^k$ with $k\ge 1$:
\[Z(k,Q_{4n};x) = O(x^{a(Q_{4n})}),\]
like expected in Conjecture \ref{con}.  
     
Using Corollary \ref{crit_prime} we can describe more cases, where our bound $d(k,G)=b(k,G)$. Let $\ell$ be the critical prime and $G_\ell$ be an $\ell$-group which fulfills the assumption of Lemma \ref{nil}. Then all direct products of the form $G_\ell \times H$, where $\ell\nmid |H|$ and $\ell$ is still the order of elements of smallest index, will fulfill the assumption of Lemma 5.13. 
Note that we get conjectured upper bound in all cases

 On the other hand the dihedral group has elements of smallest index not lying in the center, which cannot yield optimal bounds using our method. This is independent on the chosen permutation representation.

\section*{Acknowledgment}
The author would like to thank Brandon Alberts and Gunter Malle for suggestions on an earlier draft.
\bibliographystyle{plain}
\bibliography{myref}
\end{document}